%% file: ipco2019.tex
\documentclass[a4paper, 11pt]{article}
\usepackage[utf8x]{inputenc}

\usepackage{graphicx, amsfonts,amsmath,amssymb,amsthm}
\usepackage{natbib}
\usepackage{hyperref,cleveref}
\usepackage{dsfont}
\usepackage{fullpage}
\usepackage{graphicx} % more moder
\usepackage{enumitem}
\usepackage{subcaption}
\usepackage{commands}
\usepackage{multicol}
\usepackage{optidef}
\usepackage{longtable}
\usepackage{booktabs}
\usepackage{siunitx}
\begin{document}
\author{Victor Cohen$^1$ and Axel Parmentier$^1$}
\title{Linear programming for Decision Processes with Partial Information}

\date{%
    $^1$Universit\'e Paris-Est, CERMICS (ENPC), Marne-la-Vall\'ee, France\\%
    \today}

\maketitle              % typeset the header of the contribution
\begin{abstract}
Markov Decision Processes (MDPs) are stochastic optimization problems that model situations where a decision maker controls a system based on its state.
Partially observed Markov decision processes (POMDPs) are generalizations of MDPs where the decision maker has only partial information on the state of the system.
Decomposable PO\-MDPs are specific cases of POMDPs that enable one to model systems with several components. 
Such problems naturally model a wide range of applications such as predictive maintenance.
Finding an optimal policy for a POMDP is PSPACE-hard and practically challenging.
We introduce a mixed integer linear programming (MILP) formulation for POMDPs restricted to the policies that only depend on the current observation, as well as valid inequalities that are based on a probabilistic interpretation of the dependence between variables. The linear relaxation provides a good bound for the usual POMDPs where the policies depend on the full history of observations and actions.
Solving decomposable POMDPs is especially challenging due to the curse of dimensionality. 
Leveraging our MILP formulation for POMDPs, we introduce a linear program based on ``fluid formulation'' for decomposable POMDPs, that provides both a bound on the optimal value and a practically efficient heuristic to find a good policy.
Numerical experiments show the efficiency of our approaches to POMDPs and decomposable POMDPs.
% Partially Observed Markov Decision Processes (POMDP) are Markovian dynamical systems where the decision maker has only an observation about the system state. It models a larger class of real-world problems than Markov Decision Processes (MDPs). Recent literature provides exact or approximate dynamical programming algorithms to solve such problems. We propose an mixed-integer linear program solving exactly POMDP problem leveraging the usual dual linear program for MDPs, strengthened with a new class of efficient valid inequalities. However, like MDPs, when the state space and observation space grow, POMDPs become rapidly intractable. Leveraging recent literature on MDPs, we introduce \emph{decomposable POMDPs} which corresponds to a system composed by multiples components evolving individually and independently as a single POMDP. Leveraging the MILP formulation, we introduce a linear programs, that approximately solves decomposable POMDPs. The valid inequalities tighten the upper bound obtained with this linear formulation. Numerical results show the efficiency of our results.
% \keywords{Linear programming  \and Partially Observed Markov Decision Processes.}
\end{abstract}
%
%
%
\input{intro}
\input{pomdp}

\input{coupled_pomdp}

\input{num}

\section*{Acknowledgments}

We are grateful to Fr\'ed\'eric Meunier for his helpful remarks.

\bibliographystyle{plainnat}
\bibliography{coupled_pomdp}

\appendix
\input{app}





% LNCS Homepage, \url{http://www.springer.com/lncs}. Last accessed 4
% \end{thebibliography}
\end{document}

%% file: intro.tex
%!TEX root=./ipco2019.tex

\section{Introduction}
\label{sec:intro}

Many real-world problems where a decision maker controls a stochastic system evolving over time can be modeled as \emph{Partially Observed Markov Decision Processes} (POMDPs). In such problems, at each timestep, the system is in a \emph{state} $s$ in some finite state space $\calX_S$.
The decision maker does not observe $s$, but has access to an \emph{observation} $o$ that belongs to some finite observation space $\calX_O$, and is randomly emitted with probability $p(o|s)$.
Based on this observation, the decision-maker chooses an action $a$ from some finite action space $\calX_A$.
 % given an \emph{observation} $o$ from some finite observation space $\calX_O$ randomly emitted with probability $p(o|s)$. 
The system then transits randomly to a new state $s'$ in $\calX_S$ with probability $p(s'|s,a)$ and the decision maker obtains an immediate reward $r(s,a,s')$. The goal of the decision maker is to find a policy $\delta_{a | o}^t$, which represents a conditional probability of taking action $a$ in $\calX_A$ given current observation $o$ in $\calX_O$ at time $t$, maximizing the expected total reward over a finite horizon $T$
\begin{equation}\label{pb:POMDP}
    \max_{\boldsymbol{\delta} \in \Delta} \bbE_{\boldsymbol{\delta}} \bigg[ \sum_{t=1}^{T}r(S_t,A_t,S_{t+1})\bigg], 
\end{equation}
where $S_t$ and $A_t$ are random variables representing the state and the action at time $t$ and the expectation over $ \boldsymbol{\delta}$ is taken with respect to the distribution $\bbP_{\boldsymbol{\delta}}$ induced by the policy $\boldsymbol{\delta}$ chosen in the set of policies $\Delta$. In the POMDP literature, the decision maker has \emph{perfect-recall}, i.e., the decision is taken given all history of observation and actions at each time step. Hence the policy lies in a greater set of policies $\Delta_{PR} \supset \Delta$. Hence Problem \eqref{pb:POMDP} provides a lower bound on POMDP with perfect-recall. While POMDP problem with perfect-recall is PSPACE-hard \cite{Papadimitriou1987}, we restrict to Problem \eqref{pb:POMDP} in this work, which is NP-hard \cite{maua2012solving}. POMDPs are a generalization of the well-known \emph{Markov Decision Processes} (MDPs). 
POMDPs are based on Hidden Markov Models (HMMs), which give  a higher power of modeling than the usual Markov Chains (\citet{baum1966}) on which MDPs are based. 
% Indeed, the Hidden Markov Models (HMMs), which are the statistical models on which POMDPs are based, give a higher power of modeling than the usual Markov Chains \cite{baum1966} on which MDPs are based. 
They appear naturally in the context of predictive maintenance, where a machine evolves over time and at each time step the decision maker decides to replace or not the machine. The decision maker does not have access to the machine's wear and takes his action while observing output signals. Therefore, the goal is to find an optimal replacement policy (which can be time dependent) minimizing the total expected cost over a finite time horizon.

\noindent Like MDPs, POMDPs suffer from the \emph{curse of dimensionality}. Indeed, when the state space $\calX_S$ is large, the usual exact methods such as dynamic programming become computationally intractable (\citet{Puterman1994MDP}). Leveraging the well-known linear formulation of \citet{Epenoux1963} for MDPs, \citet{bertsimas2016decomposable} propose a tractable heuristic to approximately solve MDPs when the system is composed of several components evolving independently and the state space can be written as the Cartesian product of the individual state space of each component. Such problems are called \emph{decomposable MDPs}. 
We introduce the notion of \emph{decomposable POMDPs}, a generalization where each component evolves independently and individually as a POMDP. For example, in the case of predictive maintenance, a machine is composed of several equipments and each equipment can be modeled individually as a POMDP. Unfortunately, there is no tractable linear program on which we can leverage to generalize the approach of \citet{bertsimas2016decomposable}. A good reason for that is PSPACE-hardness of POMDP Problem (\citet{Papadimitriou1987}). The first exact algorithms for POMDP use dynamic programming (\citet{SondikPOMDP}) and consider the POMDP as a MDP where we replace the current state by a belief state, which corresponds to the probability of being in state $s$ at time $t$. However, such methods become computationally expensive even for small state spaces and small observation spaces. More recently, new policy-based solution methods leverage bounds to search efficiently in the policy space (\citet{Kaelbling1998}). Particularly, the finite-state controller algorithms (\citet{Boutilier2004}) combine a policy iteration algorithm that enumerates and evaluates policies on the observables history (represented as a new random variable) and gradient ascent method to find local optima, restricting the policy space. Leveraging policy search, \citet{Aras2007} proposed an exact mixed-integer linear program for POMDP problem. However, such a formulation is intractable even for small horizon time. In this paper, we consider exact and approximate linear formulations for Problem \eqref{pb:POMDP}. Our contributions are as follows :
\begin{enumerate}
\item We propose a mixed-integer linear program (MILP) that exactly solves Problem \eqref{pb:POMDP}. This formulation generalizes the dual linear program for MDP of \citet{Epenoux1963}. 
\item We introduce an extended formulation with new valid inequalities that improve the resolution of our mixed-integer linear program. Such inequalities come from a probabilistic interpretation of the dependence between random variables. Experiments show their efficiency.
\item We show that the linear relaxation of our MILP provides an upper bound on the usual POMDP with perfect-recall.
\item In the case of decomposable POMDPs, leveraging the MILP previously mentioned, we propose a heuristic that repeatedly solves linear approximations.
It extends to POMDPs to the fluid formulation introduced by \citet{bertsimas2016decomposable}, and strengthens it with valid inequalities.
\item Numerical experiments show the efficiency of our approach.
\end{enumerate}
The paper is organized as follows. Section \ref{sec:pomdp} presents our MILP formulation for Problem \eqref{pb:POMDP}, and introduces our valid inequalities. We also present the link between our formulation and the usual POMDP with perfect-recall. In Section \ref{sec:coupled_pomdp}, we define decomposable POMDPs, we present our tractable linear program giving an upper bound on the optimal value, and use it in a heuristic. In Section \ref{sec:num}, we present the numerical results. The proofs of the main results are available in Appendix \ref{sec:app}.

%% file: pomdp.tex
\section{Mixed Linear Programming for Partially Observed Markov Decision Processes}
\label{sec:pomdp}

In this section, we present an MILP that exactly models Problem \eqref{pb:POMDP}. Given $N$ in $\mathbb{N}$ we use the notation $[N]$ for $\{1, \ldots, N\}$. 
% and $\xi_{y}(x)$ denotes the indicator function of $y$ evaluated in $x$ i.e $\xi_{y}(x) = 1$ if $x=y$ and $\xi_{y}(x) = 1$ otherwise.

\subsection{Problem settings}

Let $\calX_S$, $\calX_O$, and $\calX_A$ be three finite sets corresponding respectively to the state space, the observation space and the action space.
% Moreover, we assume that any action in $\calX_A$ can be taken at any observation in $\calX_O$. An extended model of the POMDP could consider some constraints on the actions that the agent can take, depending on the observation. 
For a state $s \in \calX_S$ and an observation $o \in \calX_O$, let $p(o|s)$ be the conditional probability of observing $o$ given state $s$ 
\begin{align*}
    &p(o|s) = \mathbb{P} \left(O_t = o | S_t = s \right) \quad \forall t \in [T].
\end{align*}
Similarly, we define the probability of transition from state $s \in \calX_S$ to $s' \in \calX_S$ while taking action $a \in \calX_A$ 
\begin{align*}
    p(s') = \mathbb{P} \left(S_{1} = s'\right), \quad \text{and}, \quad p(s'|s,a) = \mathbb{P} \left(S_{t+1} = s' | S_{t} = s, A_{t} = a \right) \quad \forall t \in [T].
\end{align*}
% Remark that we suppose that the Markov process is time-homogeneous. All the contributions in this paper still hold if we relax this assumption. 
\noindent We define an immediate reward function $r : \calX_S \times \calX_A \times \calX_S \rightarrow \bbR$, which associates to a transition $(s,a,s')$ a reward $r(s,a,s')$. The goal is to solve Problem \eqref{pb:POMDP}, where
$$\displaystyle \Delta = \bigg\{ \boldsymbol{\delta} \in \bbR^{T \times |\calX_A| \times |\calX_O|}, \sum_{a \in \calX_A} \delta^t_{a|o} = 1 \ \mathrm{and} \ \delta^t_{a|o} \geq 0, \forall o \in \calX_O, a \in \calX_A, t \in [T] \bigg\},$$
and $\bbP_\delta$ and $\bbE_\delta$ denote the probability distribution and expectation induced by policy $\delta$ on $(S_t,O_t,A_t)_t$.
We define the set of \emph{deterministic} policies $\Delta^0$ as
$$\Delta^0 = \bigg\{ \boldsymbol{\delta} \in \Delta, \delta^t_{a|o} \in \{0,1\}, \forall o \in \calX_O : a \in \calX_A, t \in [T] \bigg\}.$$
Note that $\Delta$ is the convex hull of $\Delta^0$. Any policy in $\Delta \backslash \Delta^0$ is a \emph{randomized} policy.

\subsection{A mixed-integer linear program}

It is well-known that there always exists an optimal \emph{deterministic} policy for MDPs~\cite{Puterman1994MDP}. This result can be extended to POMDPs. Lemma 4.3 in \citet{liu2014reasoning} ensures that there always exists an optimal deterministic policy for Problem \eqref{pb:POMDP}. Equivalently,
\begin{equation*}
    \max_{\boldsymbol{\delta} \in \Delta} \bbE_{\boldsymbol{\delta}} \bigg[ \sum_{t=1}^{T} r(S_t,A_t, S_{t+1}) \bigg] = \max_{\boldsymbol{\delta} \in \Delta^0} \bbE_{\boldsymbol{\delta}} \bigg[ \sum_{t=1}^{T} r(S_t,A_t, S_{t+1})\bigg].
\end{equation*}
% \begin{lemma}[Lemma 4 in \citet{liu2014reasoning}]\label{lem:opt_deterministic_policy}
%     There exists an optimal policy in $\Delta^0$ for Problem \eqref{pb:decPOMDP}. Equivalently :
%     \begin{equation*}
%         \max_{\delta \in \Delta} \bbE_{\delta} \big[ \sum_{t=1}^{T} r(s_t,a_t, s_{t+1}) | s_1 = \tilde{s} \big] = \max_{\delta \in \Delta^0} \bbE_{\delta} \big[ \sum_{t=1}^{T} r(s_t,a_t, s_{t+1}) | s_1 = \tilde{s} \big]
%     \end{equation*}
% \end{lemma}
%Lemma \ref{lem:opt_deterministic_policy} ensures that we can restrict the optimization in Problem \eqref{pb:POMDP} over the set of deterministic policies.
% Let $\mu_s^t$ be the probability of being in state $s$ at time $t$, $\mu_{sa}^t$ be the probability of being in state $s$ and taking action $a$ at time $t$, and $\mu_{soa}^t$ be the probability of being in state $s$, observing $o$ and taking action $a$ at time $t$. 
% \inlAP{Introduce $\bbP_\delta$}

\noindent We introduce a collection of variables $\boldsymbol{\mu} = \big((\mu_s^t)_{s}, (\mu_{sa}^t)_{s,a}, (\mu_{soa}^t)_{s,o,a}\big)_{t}$ and the following mixed-integer linear program (MILP).
% \begin{subequations}\label{pb:NLP}
% \begin{alignat}{2}
% \max  \enskip & \sum_{t=1}^T \sum_{s,a,s'} r(s,a,s') p(s'|s,a) \mu_{sa}^t  & \quad &\\
% \mathrm{s.t.} \enskip
%  & \mu_{s'}^{t+1} = \sum_{s,a} p(s'|s,a) \mu_{sa}^t & s, s' \in \calX_S, a \in \calX_A, t \in \{1, \ldots, T\} \label{eq:NLP_consistent_state}\\
%  & \mu_{sa}^{t} = \sum_{o} p(o|s)\delta^t(a|o) \mu_s^t & s \in \calX_S, a \in \calX_A, t \in \{1, \ldots, T\} \label{eq:NLP_consistent_observation} \\
%  %& \mu_{soa}^{t} = p(o|s)\delta^t(a|o) \mu_s^t & s \in \calX_S, o \in \calX_O, a \in \calX_A, t \in \{1, \ldots, T\} \label{eq:NLP_markov_property} \\
%  & \sum_{a \in \calX_A} \delta^t(a|o) = 1 & o \in \calX_O, t \in \{1, \ldots, T\} \label{eq:NLP_normalize_policy} \\
%  & \sum_{s \in \calX_S}\mu_{s}^1 = 1 \label{eq:NLP_normalize_constraint} \\
%  & \mu_{\tilde{s}}^1 = 1 \label{eq:NLP_initial} \\
%  & \mu_s^t, \mu_{sa}^t, \mu_{soa}^t, \delta^t(a|o) \geq 0 & s \in \calX_S, o \in \calX_O, a \in \calX_A, t \in \{1, \ldots, T\} \label{eq:NLP_positivity}
% \end{alignat}
% \end{subequations}
\begin{subequations}\label{pb:MILP}
\begin{alignat}{2}
\max_{\boldsymbol{\mu}, \boldsymbol{\delta}}  \enskip & \sum_{t=1}^T \sum_{\substack{s,s' \in \calX_S \\ a \in \calX_A}} r(s,a,s') p(s'|s,a) \mu_{sa}^t  & \quad &\\
\mathrm{s.t.} \enskip
 %& \mu_{soa}^{t} = p(o|s)\delta^t(a|o) \mu_s^t & s \in \calX_S, o \in \calX_O, a \in \calX_A, t \in \{1, \ldots, T\} \label{eq:NLP_markov_property} \\
 & \mu_{s}^1 = p(s) & \forall s \in \calX_S \label{eq:MILP_initial2} \\
 & \mu_{s'}^{t+1} = \sum_{s \in \calX_S,a \in \calX_A} p(s'|s,a) \mu_{sa}^t &  \forall s' \in \calX_S, t \in [T] \label{eq:MILP_consistent_state}\\
 & \mu_{sa}^{t} = \sum_{o \in \calX_O} \mu_{soa}^t & \forall s \in \calX_S, a \in \calX_A, t \in [T] \label{eq:MILP_consistent_observation}\\
 % & \sum_{a \in \calX_A} \mu_{soa}^t = p(o|s)\mu_{s}^{t} & \forall s \in \calX_S, o \in \calX_O, t \in [T] \label{eq:MILP_consistent_observation}\\
 &\mu_{soa}^t \leq p(o|s) \mu_s^t & \forall s \in \calX_S, o \in \calX_O, a \in \calX_A, t \in [T] \label{eq:MILP_McCormick_1} \\
 &\mu_{soa}^t \leq \delta^t_{a|o} & \forall s \in \calX_S, o \in \calX_O, a \in \calX_A, t \in [T] \label{eq:MILP_McCormick_2} \\
 &\mu_{soa}^t \geq p(o|s) (\mu_s^t + \delta^t_{a|o} - 1) & \forall s \in \calX_S, o \in \calX_O, a \in \calX_A, t \in [T] \label{eq:MILP_McCormick_3} \\
 & \boldsymbol{\delta} \in \Delta^o \label{eq:MILP_constraint_policy} \\
 & \boldsymbol{\mu} \geq 0 & \label{eq:MILP_positivity}
\end{alignat}
\end{subequations}
We show in Appendix that a feasible solution $\boldsymbol{\mu}$ of Problem \eqref{pb:MILP} can be interpreted as a probability distribution over the random variables, i.e, $\mu_s^t$, $\mu_{sa}^t$ and $\mu_{soa}^t$ respectively represent the probabilities $\bbP_{\boldsymbol{\delta}}\big(S_t = s\big)$, $\bbP_{\boldsymbol{\delta}}\big(S_t = s, A_t=a\big)$ and $\bbP_{\boldsymbol{\delta}}\big(S_t = s, O_t=o, A_t=a\big)$ for all $s$ in $\calX_S$, $o$ in $\calX_O$, $a$ in $\calX_A$ and $t$ in $[T]$.
Let $v^*$ and $z^*$ respectively denote the optimal values of Problem \eqref{pb:POMDP} and Problem \eqref{pb:MILP}. The following theorem ensures that MILP~\eqref{pb:MILP} models Problem~\eqref{pb:POMDP}.

 % Problem \eqref{pb:POMDP} can be solved using our MILP.

\begin{theo}\label{theo:MILP_optimal_solution}
    Let $(\boldsymbol{\mu}, \boldsymbol{\delta})$ be a feasible solution of Problem \eqref{pb:MILP}. Then $\boldsymbol{\mu}$ is equal to the distribution $\bbP_{\boldsymbol{\delta}}$ induced by $\boldsymbol{\delta}$, and
    $(\boldsymbol{\mu}, \boldsymbol{\delta})$ is an optimal solution of Problem \eqref{pb:MILP} if, and only if $\boldsymbol{\delta}$ is an optimal deterministic policy of Problem \eqref{pb:POMDP}. Furthermore, $v^* = z^*$.
\end{theo}

\subsection{Valid inequalities}
\label{sub:validInequalities}
% Since Problem \eqref{pb:MILP} is a mixed-integer linear program, we can solve it with some linear optimization softwares using a Branch\&Bound algorithm. 
The difficulty of POMDP comes from the fact that 
\begin{equation}\label{eq:strongIndep}
\text{action variable $A_t$ is independent from state $S_t$ given observation $O_t$,} 
\end{equation}
which induces non-linearities. A corollary of these independences is that
\begin{equation}\label{eq:weakIndep}
    \text{$A_t$ is independent from $S_t$ given $O_t$, $A_{t-1}$ and $S_{t-1}$.}
\end{equation}
Theorem~\ref{theo:MILP_optimal_solution} ensures that these independences are satisfied by the distribution $\bmu$ corresponding to an integer solution $(\bmu,\bdelta)$ of MILP~\eqref{pb:MILP}.
If the component $\bmu$ of a solution $(\bmu,\bdelta)$ of the linear relaxation of MILP~\eqref{pb:MILP} can still be interpreted as a distribution, independences~\eqref{eq:strongIndep} and~\eqref{eq:weakIndep} are unfortunately no longer satisfied according to this distribution.
We now introduce an extended formulation and valid inequalities that enable to restore independences~\eqref{eq:weakIndep}.

% According to the probability distribution $\bbP_{\boldsymbol{\delta}}$ of any policy $\boldsymbol{\delta}$ in $\Delta$, the random variables $(S_t,O_t,A_t)_{1 \leq t \leq T+1}$ have to satisfy independence constraints \eqref{eq:Markov_ppty_indep}. Let $(\boldsymbol{\mu}, \boldsymbol{\delta})$ be a feasible solution of the relaxation of Problem \eqref{pb:MILP}. Since $\boldsymbol{\delta}$ is not deterministic, equality \eqref{eq:Markov_ppty_action/state} does not hold in general, i.e, constraints \eqref{eq:Markov_ppty_indep} are not satisfied in general. $\boldsymbol{\mu}$ can still be interpreted as a probability distribution , but independences \eqref{eq:Markov_ppty_indep} are not satisfied in general according to $\boldsymbol{\mu}$. Furthermore, even the slightly weaker independences
% \begin{align}\label{eq:independence_valid_cuts}
%      &A_t \indep S_t | S_{t-1}, A_{t-1}, O_{t}, \quad \forall t \in [T], 
%  \end{align}
% are not satisfied in general according to $\boldsymbol{\mu}$. Although we cannot enforce constraints \eqref{eq:Markov_ppty_indep} with linear constraints, we propose an extended formulation of Problem \eqref{pb:MILP} with valid inequalities that takes into account the weaker independences \eqref{eq:independence_valid_cuts}.
% % feasible solution of the relaxation of Problem \eqref{pb:MILP}. Since $\boldsymbol{\delta}$ is not deterministic, equality \eqref{eq:Markov_ppty_action/state} does not hold in general. Therefore, $\boldsymbol{\mu}$ is not a probability distribution over the random variables $(S_t,O_t,A_t)_{1 \leq t \leq T+1}$.
\noindent We introduce new variables $\mu_{s'a'soa}^t$ that can be interpreted as the probabilities
$$\bbP(S_{t-1} = s', A_{t-1} = a', S_t=s, O_t=o, A_t = a).$$
% e probability of being in state $s' \calX_S$ and taking action $a' \in \calX_A$ at time $t-1$, and being in state $s \in \calX_S$, observing $o \in \calX_O$ and taking action $a \in \calX_O$ at time $t$. 
Consider the following equalities
% However, in our case the linear relaxation does not take into account all independences induced by the dependence structure of the random variables.
% Indeed, if we relax the integrity constraints for the policy, then constraints \eqref{eq:MILP_McCormick_1}, \eqref{eq:MILP_McCormick_2} and \eqref{eq:MILP_McCormick_3} do not ensure that $\mu_{soa}^t = p(o|s) \delta^t(a|o)\mu_{s}^t$. 
% Equivalently, the linear relaxation does not take into account the fact that $a_t$ is conditionally independent of $s_t$ given $o_t$. 
% \noindent Enforcing this independence needs to use some non-linear constraints. However, we introduce slightly weaker independences that are not implied by the linear relaxation. By conditioning on variables $s_{t-1}$, $a_{t-1}$ and $o_t$, we enforce $s_t$ to be independent of $a_t$.
\begin{subequations}\label{eq:Valid_cuts}
    \begin{alignat}{2}
    &\sum_{s'\in \calX_S, a' \in \calX_A} \mu_{s'a'soa}^t = \mu_{soa}^{t}, &\forall s \in \calX_S, o \in \calX_O, a \in \calX_A, \label{eq:Valid_cuts_consistency1}\\
    &\sum_{a \in \calX_A} \mu_{s'a'soa}^t = p(o|s)p(s|s',a')\mu_{s'a'}^{t-1}, \quad &\forall s',s \in \calX_S, o \in \calX_O, a' \in \calX_A, \label{eq:Valid_cuts_consistency2}\\
    &\mu_{s'a'soa}^t = p(s|s',a',o)\sum_{\overline{s} \in \calX_S} \mu_{s'a'\overline{s}oa}^t, &\forall s',s \in \calX_S, o \in \calX_O, a',a \in \calX_A, \label{eq:Valid_cuts_main}
    \end{alignat}
\end{subequations}

\noindent where $$p(s|s',a',o) = \bbP(S_t=s | S_{t-1}=s',A_{t-1}=a',O_t=o).$$
Note that $p(s|s',a',o')$ does not depend on the policy $\boldsymbol{\delta}$ and can be easily computed using Bayes rules. Therefore, constraints in \eqref{eq:Valid_cuts} are linear.
\begin{prop}\label{prop:valid_cuts}
    Equalities \eqref{eq:Valid_cuts} are valid for Problem \eqref{pb:MILP}, and there exists solution $\mu$ of the linear relaxation of~\eqref{pb:MILP} that does not satisfy constraints~\eqref{eq:Valid_cuts}. 
\end{prop}

\noindent Hence, constraints \eqref{eq:Valid_cuts} strengthen the linear relaxation. Numerical experiments in Section \ref{sec:num} show the efficiency of such valid inequalities.

\begin{rem}\label{rem:numerical_aspects}
    Our extended formulation with valid inequalities has many more constraints than the initial one.
    Its linear relaxation therefore takes longer to solve.
     % by simplex algorithm since the corresponding polyhedron has a larger number of faces. Therefore, for large spaces, the computation time may be higher. 
\end{rem}

\subsection{Link with perfect-recall}
\label{subsec:link_perfect_recall}

In the previous section, we assumed the decision maker forgets all history of information. It would be interested to measure the policy obtained by solving Problem \eqref{pb:MILP} against the POMDP with perfect-recall. Let $v_{PR}^*$ be the optimal value of perfect-recall POMDP, i.e., by optimizing the expectation of \eqref{pb:POMDP} over $\Delta_{PR}$. Let $z_{LR}^*$ be the value of the linear relaxation of MILP \eqref{pb:MILP}.
\begin{theo}\label{theo:link_perfect_recall}
	The linear relaxation of MILP \eqref{pb:MILP} is an upper bound on the POMDP value with perfect-recall :
	
	\begin{equation}\label{pb:POMDP_PR}
    z^* \leq v_{PR}^* \leq z_{LR}^*
\end{equation}
\end{theo}
The proof is based on the influence diagram representation of POMDP problem. The authors in \citet[e.g. Theorem 14]{AParmentier2019} prove that the linear realaxation of our MILP is equivalent to the MDP relaxation of POMDP. Therefore, the integrity gap obtained by solving MILP \eqref{pb:MILP} bounds the gap with the value of perfect-recall POMDP. By adding equalities \eqref{eq:Valid_cuts}, we obtain a better gap.

%% file: coupled_pomdp.tex
%!TEX root=./ipco2019.tex

\section{Decomposable Partially Observed Markov Decision Processes}
\label{sec:coupled_pomdp}

Due to the curse of dimensionality, 
some applications have exponentially large $\calX_S$ and cannot be solved using MILP~\eqref{pb:MILP}.
% the size of $\calX_S$ can be  exponentially large on applications, 
% and MILP~\eqref{pb:MILP} becomes computationally intractable. % even for small spaces. 
In this section, we propose a tractable heuristic for POMDPs with  large but decomposable state space $\calX_S$.

\subsection{Problem settings}

We now introduce the notion of \emph{decomposable POMDP}. We consider a system that can be decomposed into $M$ components, so that the system state space $\calX_S$ and observation space $\calX_O$ can be written as the Cartesian product of individual state spaces $\calX_S^m$ and observation spaces $\calX_O^m$ for $m \in [M]$, i.e, $\calX_S = \calX_S^1 \times \cdots \times \calX_S^M$ and $\calX_O = \calX_O^1 \times \cdots \times \calX_O^M$. 
Let $S_t^m$ and $O_t^m$ be the random variables that represent the state and the observation of component $m \in [M]$ at time $t \in [T]$,
and $\mathbf{S}_t = \big( S_t^1, \ldots, S_t^M \big)$ and $\mathbf{O}_t = \big( O_t^1, \ldots, O_t^M \big)$ the state and observation of the complete system. 
At each time step $t$, a component $m$ is in state $S_t^m = s$, and emits an observation $O_t^m=o$ with probability $\bbP(O_t^m = o|S_t^m=s) = p_m(o|s)$.
Then, based on the observations of the full system $\mathbf{O}_t = \big( O_t^1, \ldots, O_t^M \big)$, the decision maker takes an action $A_t = a$.
Each component $m$ then evolves independently from state $S_t^m =s$ to state $S_{t+1}^{m+1}=s'$ with probability $p_m(s'|s,a)$, and the decision maker perceives reward $r_m(s^m, a, s'^m)$.

\noindent To sum things up, a decomposable POMDP is a POMDP with state space $\calX_S = \calX_S^1 \times \cdots \times \calX_S^M$, observation space $\calX_O = \calX_O^1 \times \cdots \times \calX_O^M$, action space $\calX_A$, and such that the probabilities of emission factorize as 
\begin{align*}
	&\bbP(\mathbf{O}_t = \mathbf{o} | \mathbf{S}_t = \mathbf{s}) =\prod_{m=1}^M p_m(o^m | s^m), \quad \forall t \in [T],
\end{align*}
the probabilities of transition factorize as
\begin{align*}
\bbP(\mathbf{S}_{1} = \mathbf{s}) = \prod_{m=1}^M p_m(s^m), \quad \text{and} \quad \bbP(\mathbf{S}_{t+1} = \mathbf{s'} | \mathbf{S}_t=\mathbf{s}, A_t=a) = \prod_{m=1}^M p_m(s'^m |s^m, a),
\end{align*}
and the reward decomposes additively
\begin{align*}
	r(\mathbf{s},a,\mathbf{s}') = \sum_{m=1}^M r_m(s^m, a, s'^m).
\end{align*}
The goal is still to find a policy maximizing the total expected reward
\begin{equation}\label{pb:decPOMDP}
	\max_{\boldsymbol{\delta} \in \Delta} \bbE_{\boldsymbol{\delta}} \bigg[ \sum_{t=1}^{T} \sum_{m=1}^M r_m(S_t^m,A_t, S_{t+1}^m) \bigg].
\end{equation}

\subsection{Linear formulation}

As a decomposable POMDP is a POMDP, we can solve Problem~\eqref{pb:decPOMDP} using MILP \eqref{pb:MILP}.
% Note that our MILP formulation \eqref{pb:MILP} solves Problem \eqref{pb:decPOMDP} by considering the complete system. 
However, the number of constraints and variables grows exponentially with~$M$, and even the linear relaxation of MILP \eqref{pb:MILP} becomes quickly intractable.
% has intractably many variables and constraints and even the linear relaxation of \eqref{pb:MILP} is not easy to solve. 
We propose a heuristic that repeatedly solves a tractable linear program. 
We introduce new variables $\boldsymbol{\tau} = \Big( \big( (\tau_s^{t,m})_{s}, (\tau_{sa}^{t,m})_{s,a}, (\tau_{soa}^{t,m})_{s,o,a}\big)_{m \in [M]}, \tau_a^t \Big)_{t}$ and the following linear program.
\newpage
% Let $\tilde{\mathbf{o}} \in \calX_O$ be an initial observation, we consider the following linear program :
\begin{subequations}\label{pb:decPOMDP_LP}
 \begin{flalign}
 \quad \max_{\boldsymbol{\tau}} \quad &\displaystyle \sum_{t=1}^T \sum_{m=1}^M \sum_{\substack{s^m,s'^m \in \calX_S^m \\ a \in \calX_A}} r_m(s^m,a,s'^m)p(s'^m|s^m,a)\tau_{sa}^{t,m}&
 \end{flalign}
 %& \mu_{soa}^{t} = p(o|s)\delta^t(a|o) \mu_s^t & s \in \calX_S, o \in \calX_O, a \in \calX_A, t \in \{1, \ldots, T\} \label{eq:NLP_markov_property} \\
\vspace{-0.4cm}
 \begin{alignat}{2}
 \mathrm{s.t.} \enskip &\tau_{s'}^{t+1,m} = \sum_{\substack{s \in \calX_S^m \\a \in \calX_A}} p_m(s'|s,a) \tau_{sa}^{t,m} & \forall s' \in \calX_S^m, m \in [M], t \in [T] \label{eq:LPdecPOMDPs_cons_state}\\
 &\tau_{sa}^{t,m} = \sum_{o \in \calX_O^m} \tau_{soa}^{t,m} & \forall s \in \calX_S^m, a \in \calX_A, m \in [M], t \in [T] \label{eq:LPdecPOMDPs_cons_observation} \\
 &\sum_{s \in \calX_S^m}\tau_{sa}^{t,m} = \tau_a^t & \forall a \in \calX_A, m \in [M], t \in [T] \label{eq:LPdecPOMDPs_cons_action} \\
 &\sum_{a \in \calX_A}\tau_{soa}^{t,m} = p_m(o|s)\tau_{s}^{t,m} & \forall s \in \calX_S^m, o \in \calX_O^m, m \in [M], t \in [T] \label{eq:LPdecPOMDPs_indep_observation} \\
 % &\sum_{a}\mu_{a}^{t} = 1 & 
 % t \in [T] \label{eq:decomposable_POMDPs_normalize_action} \\
 & \tau_{s}^{1,m} = p_m(s) & \forall s \in \calX_S^m, m \in [M] \label{eq:LPdecPOMDPs_initial2} \\
 &\tau \geq 0 \label{eq:LPdecPOMDPs_positivity}
 \end{alignat}
\end{subequations}
We show in Appendix that $(\tau_s^{t,m})_{s}$, $(\tau_{sa}^{t,m})_{s,a}$, $(\tau_{soa}^{t,m})_{soa}$, and $(\tau_a^t)_a$ can still be interpreted as probability distributions on $\calX_S$, $\calX_{S,m} \times \calX_{A}$, $\calX_{S,m} \times \calX_{O,m} \times \calX_{A}$, and $\calX_A$, which coincide when marginalized on the intersection of their domains.
However, there is no guarantee that there exists a joint distribution on $(\calX_S \times \calX_O \times \calX_A)^T$ from which they can be derived as marginal distributions.
%  for the variables $X_S^{tm}$, .
% % However , they amdo not guarantee the existence of a joint distributions on ... 
% Constraints \eqref{eq:LPdecPOMDPs_cons_observation}, \eqref{eq:LPdecPOMDPs_cons_action}, and \eqref{eq:LPdecPOMDPs_indep_observation} imply that $\sum_{a \in \calX_A} \tau_a^t = 1$ for any feasible solution $\boldsymbol{\tau}$ of Problem \eqref{pb:decPOMDP_LP} and for all $t \in [T]$. Therefore, $\boldsymbol{\tau}$ can be seen as a probability distribution on $\calX_A$ and a probability distribution on $\calX_S^m \times \calX_O^m \times \calX_A$ for all $m$ in $[M]$. However, it is not sufficient to ensure that $\boldsymbol{\tau}$ corresponds to the moments of a joint probability distribution on $\calX_S \times \calX_O \times \calX_A$ \citep{wainwright2008graphical}. Hence, distributions $\boldsymbol{\tau}$ are called \emph{pseudo-marginals}. 
Let $v_M^*$ and $z_M^*$ be respectively the optimal values of Problem \eqref{pb:decPOMDP} and Problem \eqref{pb:decPOMDP_LP}.
% A probability distribution $\boldsymbol{\mu}$ induced by a policy $\boldsymbol{\delta}$ of problem \eqref{pb:decPOMDP} has an exponential number of variables or marginals. In Problem \eqref{pb:decPOMDP_LP}, we restrict only to a polynomial number of variables, called \emph{pseudo-marginals}. Hence $\tau_s^{t,m}$, $\tau_{sa}^{t,m}$, $\tau_{soa}^{t,m}$ and $\tau_a^t$ respectively denote the pseudo-marginals of $S_t^m$, $\big( S_t^m, A_t \big)$, $\big( S_t^m, O_t^m, A_t \big)$ and $A_t$.
% Furthermore, the number of constraints is polynomial while Problem \eqref{pb:decPOMDP} has an exponential number of constraints when we use formulation \eqref{pb:MILP}. 
% The next theorem states that Problem \eqref{pb:decPOMDP_LP} provides an upper bound on the value of Problem \eqref{pb:decPOMDP}.

\begin{theo} \label{theo:relaxation_dec_POMDPs}
	Problem \eqref{pb:decPOMDP_LP} is a relaxation of Problem \eqref{pb:decPOMDP}, and $v^*_M \leq z^*_M.$
\end{theo}
This inequality is not an equality in general, as we will see in the numerical experiments.
 % Indeed, as shown in Section \ref{sec:num}, it is possible to construct a simple numerical counter-example showing that the inequality is not tight in general.
Linear formulation \eqref{pb:decPOMDP_LP} is a generalization to decomposable POMDPs of the ``fluid'' formulation proposed by \citet{bertsimas2016decomposable} for decomposable MDPs. 

% However, when the linear formulation of \citet{bertsimas2016decomposable} enabled to obtain a good quality bound on the value of decomposable MDP, $z^*_M$ provides a poor upper bound. A good upper bound can be obtained by introducing the valid equalities of Section \ref{sec:pomdp}. Since equalities \eqref{eq:Valid_cuts} are valid for Problem \eqref{pb:decPOMDP}, we can naturally derive similar equalities to tighten relaxation \eqref{pb:decPOMDP_LP}. 
The quality of the bound $z^*_M$ can be improved by generalizing the valid inequalities introduced in Section~\ref{sub:validInequalities}.
We introduce new variables $\tau_{s'a'soa}^{t,m}$ that can be interpreted as the probability $$\bbP(S_{t-1}^m=s', A_{t-1}=a',S_t^m=s,O_t^m=o,A_t=a).$$ 
Consider the following linear inequalities.
% We consider the following inequalities for all $m \in [M]$ and $t \in [T]$
% \newpage 
\begin{subequations}\label{eq:dec_Valid_cuts}
    \begin{alignat}{2}
    &\sum_{s'\in \calX_S, a' \in \calX_A} \tau_{s'a'soa}^{t,m} = \tau_{soa}^{t,m},& \forall s \in \calX_S^m, o \in \calX_O^m, a \in \calX_A, \label{eq:dec_Valid_cuts_consistency1} \\
    &\sum_{a \in \calX_A} \tau_{s'a'soa}^{t,m} = p_m(o|s)p_m(s|s',a')\tau_{s'a'}^{t-1,m}, & \forall s,s' \in \calX_S^m, o \in \calX_O^m, a' \in \calX_A, \label{eq:dec_Valid_cuts_consistency2} \\
    &\tau_{s'a'soa}^{t,m} = p_m(s|s',a',o)\sum_{\overline{s} \in \calX_S} \tau_{s'a'\overline{s} oa}^{t,m}, \quad \quad &\forall s,s' \in \calX_S^m, o \in \calX_O^m,a,a' \in \calX_A. \label{eq:dec_Valid_cuts_main}
    \end{alignat}
\end{subequations} 
\begin{prop}
Theorem~\ref{theo:relaxation_dec_POMDPs} remains true if we add equalities~\eqref{eq:dec_Valid_cuts} to \eqref{pb:decPOMDP_LP} for all $m \in [M]$ and $t \in [T]$. 
There are solutions of~\eqref{pb:decPOMDP_LP} that do not satisfy \eqref{eq:dec_Valid_cuts}.
\end{prop}

Since decomposable POMDPs are POMDPs, we would like to compare our policies with optimal value of POMDP with perfect-recall. Let $v_{PR}^*$ be the optimal value of POMDP with perfect-recall, i.e., the value of the expectation in \eqref{pb:decPOMDP} by optimizing over $\Delta_{PR}$.
\begin{theo}\label{theo:link_perfect_recall}
	The linear formulation \eqref{pb:decPOMDP_LP} gives an upper bound on the POMDP value with perfect-recall, i.e., $v_{PR}^* \leq z_{M}^*$.	
\end{theo}

\subsection{Heuristic}

Consider an optimal solution $\boldsymbol{\tau}$ of Problem \eqref{pb:decPOMDP_LP}. Solution $\boldsymbol{\tau}$ is \emph{achievable} if there exists a policy $\boldsymbol{\delta} \in \Delta$ such that
\begin{align}\label{eq:sufficient_condition}
	&\delta^t_{a|\textbf{o}} = \frac{\tau_{soa}^{t,m}}{\displaystyle \sum_{\overline{a} \in \calX_A} \tau_{so\overline{a}}^{t,m}}, && \forall s \in \calX_S^m, \mathbf{o} \in \calX_O, a \in \calX_A, m \in [M], t \in [T]. 
\end{align}
The following proposition generalizes Proposition 3 of \citet{bertsimas2016decomposable}.

\begin{prop}\label{prop:equality_dec_POMDPs}
	Let $\boldsymbol{\tau}$ be an optimal solution of Problem \eqref{pb:decPOMDP_LP}.
	If $\boldsymbol{\tau}$ is achievable, then
	 $$v_M^*=z_M^*.$$
\end{prop}
\noindent In the remainder of this section, we assume that the decision maker observes the initial observation $\tilde{\mathbf{o}}$ in $\calX_O$. This assumption imposes conditioning all probabilities on the event $\big\{\mathbf{O}_1 = \tilde{\mathbf{o}} \big\}$ in Problem \eqref{pb:decPOMDP} and Problem \eqref{pb:decPOMDP_LP}. Therefore, it suffices to replace $p_m(o|s)$ by $\mathds{1}_{\tilde{o}^m}(o)$ at time $t=1$
% we replace constraints \eqref{eq:LPdecPOMDPs_indep_observation} for $t=1$ in Problem \eqref{pb:decPOMDP_LP} by
% \begin{align}\label{eq:LPdecPOMDPs_init_observation}
%  	&\sum_{a \in \calX_A}\tau_{soa}^{1,m} = \mathds{1}_{\tilde{o}^m}(o)\tau_{s}^{1,m}, & \forall s \in \calX_S^m, o \in \calX_O^m, m \in [M],
%  \end{align}
where $\mathds{1}_{y}(x)$ is the indicator function of $y$ evaluated in $x$.
% Under Condition \eqref{eq:sufficient_condition}, we can construct an optimal deterministic policy for the initial action at $t=1$.

\begin{theo}\label{theo:Initial_deterministic_pol}
	Let $\boldsymbol{\tau}$ be an optimal solution of Problem \eqref{pb:decPOMDP_LP} with an initial observation $\tilde{\mathbf{o}} \in \calX_O$. Suppose that $\boldsymbol{\tau}$ is achievable and let $\boldsymbol{\delta}$ be the policy achieving $\boldsymbol{\tau}$. We define the initial policy $\overline{\delta}^1$ :
	\begin{align}\label{eq:opt_init_policy}
		&\overline{\delta}_{a^*|\tilde{\mathbf{o}}}^1 = 1 \ \mathrm{if} \ a^* = \argmax_{a \in \calX_A} \tau_a^1
	\end{align}
	Then the policy $(\overline{\delta}^1, \delta^2, \ldots, \delta^T)$ is optimal for Problem \eqref{pb:decPOMDP}.
\end{theo}

\noindent Under condition \eqref{eq:sufficient_condition}, Theorem \ref{theo:Initial_deterministic_pol} says that we can take an optimal deterministic action at time $t=1$. However, an optimal solution $\boldsymbol{\tau}$ of Problem \eqref{pb:decPOMDP_LP} is not always achievable. Nevertheless, a feasible solution of Problem \eqref{pb:decPOMDP_LP} satisfying the valid inequalities \eqref{eq:dec_Valid_cuts} respects the transition of all independent components and almost all independences between random variables. Therefore, $\boldsymbol{\tau}$ may be close to being achievable. Consequently, we propose a heuristic similar to the algorithm presented in \citet{bertsimas2016decomposable}, that repeatedly solves Problem \eqref{pb:decPOMDP_LP} and selects action $\argmax_{a \in \calX_A} \tau_a^1$ (that depends on initial observation $\tilde{\mathbf{o}}$). 
Algorithm~\ref{alg:heuristic} states our heuristic policy, which we expect to provide good performances.
% Since we have a linear formulation with a polynomial number of constraints and variables, we can propose a heuristic to approximately solve Problem \eqref{pb:decPOMDP} with a polynomial complexity. Indeed 

\begin{algorithm}[H]
\caption{Heuristic policy for Problem \eqref{pb:decPOMDP}}
\label{alg:heuristic}
\begin{algorithmic}[1]
\STATE \textbf{Input} $T$, $p_m$, $r_m$ for all $m \in [M]$, current observation $\tilde{\mathbf{o}} \in \calX_O$
\STATE Solve Problem \eqref{pb:decPOMDP_LP} with initial observation $\tilde{\mathbf{o}}_t$ to obtain an optimal solution $\boldsymbol{\tau}$.
\STATE Take action $a^* = \argmax \tau_{a}^1$.
\end{algorithmic}
\end{algorithm}

\noindent Note that each iteration of Algorithm \ref{alg:heuristic} solves a linear program with a polynomial number of constraints and variables. Numerical experiments in Section \ref{sec:num} show the efficiency of this heuristic.

%% file: num.tex
%!TEX root=./ipco2019.tex

\section{Numerical experiments}
\label{sec:num}

We now provide experiments showing the practical efficiency of our approaches to POMDPs and decomposable POMDPs.
All linear programs have been implemented in Julia with JuMP interface and solved using Gurobi 7.5.2. Experiments have been run on a server with 192Gb of RAM and 32 cores at 3.30GHz.

%%%%%%%%%%%%%%%%%%%%%%%%%%%%%%%%%%%%%%%%%%%%%%%%%%%%%%%%%%%%%%%%%%%%%%%%%%%%%%%
\subsection{Instances}
\label{sub:instances}
%%%%%%%%%%%%%%%%%%%%%%%%%%%%%%%%%%%%%%%%%%%%%%%%%%%%%%%%%%%%%%%%%%%%%%%%%%%%%%%
Each instance is generated by first choosing $\big \vert \calX_S \big \vert$, $\big \vert \calX_O \big \vert$, $\big \vert \calX_A \big \vert$, and finally $M$ when we consider decomposable MDPs.
We then randomly generate the initial probability $p_m(s)$, the transition probability $p_m(s'|s,a)$, the emission probability $p_m(o|s)$ and the immediate reward function $r_m(s,a,s')$ for all $m \in [M]$.
% by solving linear programs on a rolling horizon in Algorithm~\ref{alg:heuristic}.

\begin{rem}
In our numerical experiments, we use instances with short horizons $T$ in $\{5,10,20\}$. This is not a limitation in practice, as the length $T$ of the horizon is not the main challenge when designing heuristics for decision processes. 
\citet{bertsimas2016decomposable} solve large or infinite horizon instances of decomposable MDPs using rolling horizon approaches. 
They obtain good performances by solving at each time steps problems on a horizon  $T = 10$. 
Similar rolling techniques can be used to adapt Algorithm~\ref{alg:heuristic} to  large or infinite horizon decomposable POMDPs. 
\end{rem}

\subsection{Simulated experiments on single POMDP}

We solve Problem \eqref{pb:MILP} with and without valid equalities \eqref{eq:Valid_cuts}. Algorithms were stopped after a Time Limit (TL) of $600$s. Table \ref{tab:random_instances_results} shows the efficiency of MILP~\eqref{pb:MILP} to solve \eqref{pb:POMDP}. The first four columns indicate the size of state space $\big\vert \calX_{S}\big\vert$, observation space $\big\vert \calX_{O}\big\vert$, action space $\big\vert \calX_{A}\big\vert$ and time horizon $T$. The fifth column indicates the mathematical program used to solve Problem \eqref{pb:MILP} with or without constraints \eqref{eq:Valid_cuts}. In the last three columns, we report the integrity gap, the final gap and the computation time for each instance. 
% Table \ref{tab:random_instances_results} shows the efficiency of Problem \eqref{pb:MILP} to solve \eqref{pb:POMDP}. Indeed, the integrity gap is always significantly lower for MILP-cuts, which shows the tightening of linear relaxation and it greatly reduces the computation time.

\begin{table}
  \begin{tabular}{|SSS|c|c|S|SSc|}
        \hline
        \multicolumn{1}{|c}{$\big\vert \calX_{S}\big\vert$} & \multicolumn{1}{c}{$\big\vert \calX_{O} \big\vert$} & \multicolumn{1}{c}{$\big\vert \calX_{A} \big\vert$} & \multicolumn{1}{c}{$T$} & \multicolumn{1}{c}{Nb. Policies} & \multicolumn{1}{c}{Prog.} & \multicolumn{1}{c}{Int. Gap (\%)} & \multicolumn{1}{c}{Final Gap (\%)} & \multicolumn{1}{c}{Time (s)}  \vline\\ \hline
        3 & 3 & 3 & 10 & \text{${10^{14}}$} & \eqref{pb:MILP} & 3.57 & Opt &  0.85 \\ \cline{6-9}
          &   &   &    &                    & \eqref{pb:MILP} \text{and} \eqref{eq:Valid_cuts} & 0.43 & Opt &  0.16 \\ \cline{4-9}
          &   &   & 20 & \text{${10^{28}}$} & \eqref{pb:MILP} & 3.53 & Opt &  11.79 \\ \cline{6-9}
          &   &   &    &                    &\eqref{pb:MILP} \text{and} \eqref{eq:Valid_cuts} & 0.22 & Opt & 0.46 \\ \hline
        3 & 4 & 4 & 10 & \text{${10^{24}}$} & \eqref{pb:MILP} & 2.74 & Opt &  1.82  \\ \cline{6-9}
          &   &   &    &                    & \eqref{pb:MILP} \text{and} \eqref{eq:Valid_cuts} & 0.61 & Opt &  1.35\\\cline{4-9}
          &   &   & 20 & \text{${10^{48}}$} & \eqref{pb:MILP} & 2.67 & Opt &  326.71 \\  \cline{6-9}
          &   &   &    &                    & \eqref{pb:MILP} \text{and} \eqref{eq:Valid_cuts} & 0.53 & Opt & 5.85 \\ \hline
        3 & 5 & 5 & 10 & \text{$10^{34}$} & \eqref{pb:MILP} & 8.88 & Opt & 12 \\ \cline{6-9}
          &   &   &    &                  & \eqref{pb:MILP} \text{and} \eqref{eq:Valid_cuts} & 2.61 & Opt &  3.58 \\\cline{4-9}
          &   &   & 20 & \text{${10^{69}}$} & \eqref{pb:MILP} & 9.06 & 2.56 & TL \\  \cline{6-9}
          &   &   &    &                    & \eqref{pb:MILP} \text{and} \eqref{eq:Valid_cuts} & 2.45 & Opt & 116.54 \\ \hline
        4 & 8 & 8 & 10 & \text{${10^{72}}$} & \eqref{pb:MILP} & 15.50 & 8.64 & TL \\ \cline{6-9}
          &   &   &    &                    & \eqref{pb:MILP} \text{and} \eqref{eq:Valid_cuts} & 1.77 & Opt &  383.32\\ \cline{4-9}
          &   &   & 20 & \text{${10^{144}}$} & \eqref{pb:MILP} & 15.64 & 12.57 & TL \\ \cline{6-9}
          &   &   &    &                     & \eqref{pb:MILP} \text{and} \eqref{eq:Valid_cuts} & 1.29 & 0.62 & TL \\\hline
        %   &   &   & 10 & 14.4 & 8.67 & TL & 1.89 & 0.00 & 1.07e+3  \\
        %   &   &   & 20 & 15.6 & 14.7 & TL & 2.02 & 1.41 & TL  \\
        %   & 5 & 5 & 5 & 288 & \color{red}238 & TL & 16 & \color{blue}11 & TL \\
        % 10  & 4 & 4 & 4 & 12.0 & 0.9 & TL & 2.4 & 0.0 &  5.1  \\
        %   & 5 & 5 & 5 & 7.0 &  1.3 & TL & 0.7 & 0.0 &  5.7 \\ \midrule
        % 20  & 4 & 4 & 4 & 9.40 & 6.34 &  TL & 1.41 & 0.44 & TL  \\
        %     & 5 & 5 & 5 & 11.8 & 9.48 &  TL & 1.15 & 0.72 & TL \\
        % \hline
  \end{tabular}
 \caption{POMDP results using MILP \eqref{pb:MILP} with and without \eqref{eq:Valid_cuts}, with a time limit TL=$3600$s}
 \label{tab:random_instances_results}
 % \caption{My table}
% \label{table:kysymys}
 \end{table}

\subsection{Simulated experiments on decomposable POMDPs}

We now compare two heuristics to solve decomposable POMDPs: 
% First, Algorithm~\ref{alg:heuristic} which Problem \eqref{pb:decPOMDP_LP} with valid inequalities~\eqref{eq:dec_Valid_cuts}. second, 
Algorithm \ref{alg:heuristic} solving Problem \eqref{pb:decPOMDP_LP} without valid inequalities~\eqref{eq:dec_Valid_cuts}, and a greedy algorithm, which infers the most probable state and takes the action maximizing the expected immediate reward of the next state. 
On each instance, 
we compute an upper bound $z_{M,k}^*$ on the optimal value by solving Problem \eqref{pb:decPOMDP_LP} with constraints \eqref{eq:dec_Valid_cuts}.
We test our two heuristics under $K=100$ different scenarios.

%  generate $K=100$ initial random pairs $(\mathbf{s}^{(k)}, \mathbf{o}^{(k)})$ 
% and we run the three heuristics over time horizon $T$. 

Table \ref{tab:random_instances_results_heuristic} summarizes the results obtained. 
The first five columns describe the instance. They indicate the sizes of state space $\big\vert \calX_{S}\big\vert$, observation space $\big\vert \calX_{O}\big\vert$,  and action space $\big\vert \calX_{A}\big\vert$, the horizon time $T$, and the number of components $M$. 
The sixth column indicates the heuristic used. 
The next column provides the average computation time needed at each time step to take a decision. 
When using Algorithm~\ref{alg:heuristic}, this is the time needed to solve MILP~\eqref{pb:decPOMDP_LP}.
Finally, column ``Av.~gap ($\%$)'' provides the gap
% We define the gap value 
$$100 \times \frac{z_{M,k}^* - R_{k,h}}{z_{M,k}^*},$$ 
between the upper bound $z_{M,k}^*$ and total reward $R_{k,h}$ obtained using the heuristic on average on the $K$ scenarios.
% where $R_{k,h}$ is the reward obtained for scenario $k$ when we run heuristic $h$ and $z_{M,k}^*$ is the upper bound obtained by solving Problem \eqref{pb:decPOMDP_LP} with constraints \eqref{eq:dec_Valid_cuts} and initial observation $\mathbf{o}^k$ at time $t=1$. The two last columns of Table \ref{tab:random_instances_results_heuristic} report for each heuristic, the average of the computation times for one policy over all scenarios and the average of gap values $G_{mean,h}$ over all scenarios.

\begin{table}[!ht]
  \begin{tabular}{|S|SSS|S|S|S|S|}
        \hline
        \multicolumn{1}{|c}{$M$} & \multicolumn{1}{c}{$\big\vert \calX_{S}\big\vert$} & \multicolumn{1}{c}{$\big\vert \calX_{O}\big\vert$} & \multicolumn{1}{c}{$\big\vert \calX_{A} \big\vert$} & \multicolumn{1}{c}{$T$} & \multicolumn{1}{c}{Heuristic (h)} & \multicolumn{1}{c}{CPU time (s)} & \multicolumn{1}{c|}{Av.~gap ($\%$)}  \\ \hline
        % {$M$} & {$\big\vert \calX_{S}\big\vert$} & {$\big\vert \calX_{O} \big\vert$} & {$\big\vert \calX_{A} \big\vert$} & {$T$} & {Heuristics} & {CPU time (s)} & {Mean} \\ \midrule
        3  & 5 & 5 & 5 & 5 & \text{Greedy} & 0.00 & 22.43  \\
          &  &  &  &  & \text{Alg. \eqref{alg:heuristic}} & 2.52$\rm{e}$-2 & 4.95 \\ \cline{5-8}
         & &  &  & 10 & \text{Greedy} & 0.00 & 21.17 \\
          &   &   &   &   & \text{Alg. \eqref{alg:heuristic}} & 8.07$\rm{e}$-2 & 4.00  \\  \hline
        4  & 3 & 5 & 5 & 5 & \text{Greedy} & 0.00 & 18.53  \\
          &  &  &  &  & \text{Alg. \eqref{alg:heuristic}} & 2.17$\rm{e}$-2 & 8.73 \\ \cline{5-8}
         & &  &  & 10 & \text{Greedy} & 0.00 & 18.40  \\
          &   &   &   &   & \text{Alg. \eqref{alg:heuristic}} & 7.66$\rm{e}$-2 & 7.97  \\  \hline
        5  & 5 & 5 & 5 & 5 & \text{Greedy} & 0.00 & 17.3  \\
          &  &  &  &  & \text{Alg. \eqref{alg:heuristic}} & 5.35e-2 & 7.24 \\ \cline{5-8}
         & &  &  & 10 & \text{Greedy} & 0.00 & 17.9  \\
          &   &   &   &   & \text{Alg. \eqref{alg:heuristic}} & 1.73e-1 & 6.11  \\  \hline
        10 & 5 & 5 & 5 & 5 & \text{Greedy} & 0.00 & 14.7  \\
          &  &  &  &  & \text{Alg. \eqref{alg:heuristic}} & 1.69e-1 & 5.01 \\
        %   & 5 & 5 & 5 & 288 & \color{red}238 & TL & 16 & \color{blue}11 & TL \\
        \hline
    \end{tabular}
 \caption{Heuristic performances on decomposable POMDPs ($M>1$)}
 \label{tab:random_instances_results_heuristic}
 % \caption{My table}
% \label{table:kysymys}
 \end{table}
 
 \noindent Note that when the number of components grows, Algorithm~\ref{alg:heuristic} outperforms the standard greedy algorithm.

 % \begin{rem}
 % An improved heuristic can be obtained using LP~\eqref{pb:decPOMDP_LP} with valid inequalities \eqref{eq:dec_Valid_cuts} in Algorithm~\ref{alg:heuristic}.
 % Such an approach is tractable in practice as solving the resulting linear program on each instance takes at most 10 minutes. 
 % However, evaluating such policies on 100 scenarios takes times. We are currently performing numerical experiments with this heuristic, and will present the numerical results at the conference.
 % \end{rem}

%% file: app.tex
%!TEX root=./ipco2019.tex

\section{Proofs}
\label{sec:app}

\begin{proof}[Proof of Theorem \ref{theo:MILP_optimal_solution}]
	% We consider the following non-linear program :
	% \begin{subequations}\label{pb:NLP}
	% \begin{alignat}{2}
	% \max_{\boldsymbol{\mu}}  \enskip & \sum_{t=1}^T \sum_{s,a,s'} r(s,a,s') p(s'|s,a) \mu_{sa}^t  & \quad &\\
	% \mathrm{s.t.} \enskip
 % 	& \mu_{s'}^{t+1} = \sum_{s,a} p(s'|s,a) \mu_{sa}^t & s, s' \in \calX_S, a \in \calX_A, t \in [T] \label{eq:NLP_consistent_state}\\
 % 	& \mu_{sa}^{t} = \sum_{o} \mu_{soa}^{t} & s \in \calX_S, a \in \calX_A, t \in [T] \label{eq:NLP_consistent_observation} \\
 % 	& \mu_{soa}^{t} = p(o|s)\delta^t_{a|o} \mu_s^t & s \in \calX_S, o \in \calX_O, a \in \calX_A, t \in [T] \label{eq:NLP_markov_property} \\
 % 	& \mu_{s}^1 = p(s) & s \in \calX_S \label{eq:NLP_normalize_constraint} \\
 % 	& \delta \in \Delta \label{eq:NLP_constraint_policy}\\
 % 	& \boldsymbol{\mu} \geq 0 & s \in \calX_S, o \in \calX_O, a \in \calX_A, t \in [T] \label{eq:NLP_positivity}
	% \end{alignat}
	% \end{subequations}
	Let $(\boldsymbol{\mu},\boldsymbol{\delta})$ be a feasible solution of Problem \eqref{pb:MILP}. 
	We prove by induction on $t$ that $\mu_s^t = \bbP_{\boldsymbol{\delta}}\big(S_t = s\big)$, $\mu_{sa}^t=\bbP_{\boldsymbol{\delta}}\big(S_t = s, A_t=a\big)$ and $\mu_{soa}^t =\bbP_{\boldsymbol{\delta}}\big(S_t = s, O_t=o, A_t=a\big)$. 
	At time $t=0$, the statement is immediate. Suppose that it holds up to $t-1$. Constraints \eqref{eq:MILP_consistent_state} and induction hypothesis ensure that 
	\begin{align*}
	\mu_s^{t} &= \sum_{s',a'} p(s|s',a') \bbP_{\boldsymbol{\delta}}\big(S_{t-1} = s', A_{t-1} = a \big) \\
			  &= \bbP_{\boldsymbol{\delta}}\big(S_{t} = s\big)
	\end{align*}
	where the last equality comes from the law of total probability. Since $\delta_{a|o}^t$ are binary, constraints \eqref{eq:MILP_McCormick_1}, \eqref{eq:MILP_McCormick_2} and \eqref{eq:MILP_McCormick_3} ensure that :
	\begin{align*}
		\mu_{soa}^t &= \delta_{a|o}^t p(o|s)\mu_s^t \\
					&= \bbP_{\boldsymbol{\delta}}(S_t=s, O_t=o, A_t=a).
	\end{align*}
	Finally, constraints \eqref{eq:MILP_consistent_observation} ensure that $\mu_{sa}^t = \bbP_{\boldsymbol{\delta}}(S_t=s, A_t=a)$. Therefore, any feasible solution $\boldsymbol{\mu}$ of \eqref{pb:MILP} is equal to the distribution $\bbP_{\boldsymbol{\delta}}$.
	Consequently,
	$$\sum_{t=1}^T \sum_{\substack{s,s' \in \calX_S \\ a \in \calX_A}} r(s,a,s') p(s'|s,a) \mu_{sa}^t = \bbE_{\boldsymbol{\delta}} \bigg[ \sum_{t=1}^{T}r(S_t,A_t,S_{t+1})\bigg],$$
	which implies that $\boldsymbol{\delta}$ is optimal if and only if $(\boldsymbol{\mu},\boldsymbol{\delta})$ is optimal for Problem \eqref{pb:MILP} and $v^* = z^*$.

	% We prove that Problem \eqref{pb:POMDP} is equivalent to Problem \eqref{pb:NLP}.	
	% Let $\boldsymbol{\delta}$ be a feasible solution of \ref{pb:POMDP}. By induction, we build the variables $\mu_s^t$ and $\mu_{sa}^t$ using formula of constraints \eqref{eq:NLP_consistent_state} and \eqref{eq:NLP_consistent_observation}. The solution $\boldsymbol{\mu}$ is feasible for Problem \eqref{pb:MILP} and the expected rewards are equals.

	% \noindent Conversely, let $(\boldsymbol{\mu}, \boldsymbol{\delta})$ be a feasible solution of Problem \eqref{pb:NLP}. Therefore $\boldsymbol{\delta} \in \Delta$ and we have the same expected reward. Since we can restrict to the subset of deterministic policies, we replace constraints \eqref{eq:NLP_constraint_policy} by constraints \eqref{eq:MILP_constraint_policy}. 

	% \noindent To conclude, we observe that the set of constraints \eqref{eq:MILP_McCormick_1}, \eqref{eq:MILP_McCormick_2} and \eqref{eq:MILP_McCormick_3} is equivalent to constraints \eqref{eq:NLP_markov_property} for deterministic policies. It achieves the proof.
\end{proof}

\begin{proof}[Proof of the validity of equalities\eqref{eq:Valid_cuts}]
	Let $(\boldsymbol{\mu}, \boldsymbol{\delta})$ be a feasible solution of problem \eqref{pb:MILP}. We define
	$$\mu_{s'a'soa}^t = \delta^t_{a|o}p(o|s)p(s|s',a')\mu_{s'a'}^{t-1}$$ for all $(s',a',s,o,a) \in \calX_S \times \calX_A \times \calX_S \times \calX_O \times \calX_A$, $t \in [T]$. These new variables satisfy constraints in \eqref{eq:Valid_cuts} :
	\begin{align*}
		\sum_{a \in \calX_A} \mu_{s'a'soa}^t 
		&= (\sum_{a \in \calX_A} \delta^t_{a|o})p(o|s)p(s|s',a')\mu_{s'a'}^{t-1} \\
		&= p(o|s)p(s|s',a')\mu_{s'a'}^{t-1}\\
		\sum_{a' \in \calX_A, s' \in \calX_S} \mu_{s'a'soa}^t &= (\sum_{a' \in \calX_A, s' \in \calX_S} p(s|s',a')\mu_{s'a'}^{t-1}) \delta^t_{a|o} p(o|s) \\
						&= \mu_{s}^{t} \delta^t_{a|o} p(o|s) \\
						&= \mu_{soa}^t
	\end{align*}

	\noindent The remaining constraint \eqref{eq:Valid_cuts_main} is obtained using the following observation :
	\begin{align*}
		\frac{\mu_{s'a'soa}^t}{\displaystyle \sum_{\overline{s} \in \calX_S} \mu_{s'a'\overline{s}oa}^t} = \frac{p(o|s)p(s|s',a')}{\displaystyle \sum_{\overline{s} \in \calX_S} p(o|\overline{s})p(\overline{s}|s',a')}
	\end{align*}
	By setting $p(s|s',a',o) = \frac{\displaystyle p(o|s)p(s|s',a')}{\displaystyle \sum_{\overline{s} \in \calX_S} p(o|\overline{s})p(\overline{s}|s',a')}$, equality \eqref{eq:Valid_cuts_main} holds. 	
\end{proof}

% \newpage
\subsection{Decomposable Partially Observed Markov Decision Processes}

\begin{proof}[Proof of Theorem \ref{theo:relaxation_dec_POMDPs}]
	We prove that an optimal solution $(\boldsymbol{\mu}, \boldsymbol{\delta})$ of Problem \eqref{pb:MILP} is a feasible solution of Problem \eqref{pb:decPOMDP_LP}. For each $m \in [M]$, we define $\displaystyle \tau_s^{t,m} = \sum_{\textbf{s}^{-m} \in \prod_{j \neq m} \calX_S^j} \mu_{\textbf{s}}^t$ where $\textbf{s}^{-m}$ is the vector $\textbf{s}$ without the $m$-th coordinate corresponding to component $m$. Similarly, we define $\displaystyle \tau_{sa}^{t,m} = \sum_{\textbf{s}^{-m} \in \prod_{j \neq m} \calX_S^j} \mu_{\textbf{s}a}^t$, $\displaystyle \tau_{soa}^{t,m} = \sum_{\substack{\textbf{s}^{-m} \in \prod_{j \neq m} \calX_S^j \\ \textbf{o}^{-m} \in \prod_{j \neq m} \calX_O^j}} \mu_{\textbf{so}a}^t$ and $\displaystyle \mu_a^t = \sum_{\textbf{s} \in \calX_S} \mu_{\textbf{s}a}^t$. This solution is indeed a feasible solution of Problem \eqref{pb:decPOMDP_LP}.
\end{proof}

\begin{proof}[Proof of Proposition \ref{prop:equality_dec_POMDPs}]
	We prove that we can build a feasible solution of Problem \eqref{pb:MILP} for $M$ components. We build such a solution by induction on $t$. For $t=1$, we define $\mu_{\textbf{s}}^1 = \prod_{m=1}^M \tau_{s}^{1,m}$,  $\mu_{\textbf{so}}^1 = p(\textbf{o}|\textbf{s})\mu_{\textbf{s}}^1$, $\mu_{\textbf{soa}}^1 = \delta^1_{a|\textbf{o}}\mu_{\textbf{so}}^1$ and $\mu_{\textbf{s}a}^1 = \sum_{\textbf{o} \in \calX_O} \mu_{\textbf{so}a}^1$. We define the following induction equation :
	\begin{align*}
		&\mu_{\textbf{s'}}^{t+1} = \sum_{\textbf{s} \in \calX_S} \sum_{a \in \calX_A} p(\textbf{s'}|\textbf{s},a) \mu_{\textbf{s}a}^t \\
		&\mu_{\textbf{so}}^{t+1} = p(\textbf{o}|\textbf{s})\mu_{\textbf{s}}^{t+1} \\
		&\mu_{\textbf{so}a}^{t+1} = \delta^{t+1}_{a|\textbf{o}}\mu_{\textbf{so}}^{t+1} \\
		&\mu_{\textbf{s}a}^{t+1} = \sum_{\textbf{o} \in \calX_O} \mu_{\textbf{so}a}^{t+1}
	\end{align*}
	It is easy to observe that all constraints of problem \eqref{pb:MILP} are satisfied. Therefore, we build a feasible solution of Problem \eqref{pb:MILP} with the same expected reward. Since Problem \eqref{pb:decPOMDP} can be exactly solved by Problem \eqref{pb:MILP} by considering the complete system and the probabilities over $\calX_S$ and $\calX_O$. Using Proposition \eqref{theo:relaxation_dec_POMDPs}, there is equality $z_M = v_M$. 
\end{proof}

\begin{proof}[Proof of Theorem \eqref{theo:Initial_deterministic_pol}]
	We define $\boldsymbol{\delta}^* = (\overline{\delta}, \delta^2, \ldots, \delta^T)$. We prove that :
	\begin{align*}
		\bbE_{\boldsymbol{\delta}} \bigg[ \sum_{t=1}^T r(\mathbf{S}_t,A_t,\mathbf{S}_{t+1}) | \mathbf{o}_1 = \tilde{\mathbf{o}} \bigg] = \bbE_{\boldsymbol{\delta}^*} \bigg[\sum_{t=1}^T r(\mathbf{S}_t,A_t,\mathbf{S}_{t+1}) | \mathbf{o}_1 = \tilde{\mathbf{o}}\bigg]
	\end{align*}
	We have :
	\begin{align*}
		\bbE_{\boldsymbol{\delta}} \bigg[ \sum_{t=1}^T r(\mathbf{S}_t,A_t,\mathbf{S}_{t+1}) | \mathbf{o}_1 = \tilde{\mathbf{o}}\bigg] &= \bbE_{\boldsymbol{\delta}} \bigg[r(s_1,a_1,s_2) | \mathbf{o}_1 = \tilde{\mathbf{o}}\bigg] + \bbE_{\boldsymbol{\delta}} \bigg[ \sum_{t=2}^T r(\mathbf{S}_t,A_t,\mathbf{S}_{t+1}) | \mathbf{o}_1 = \tilde{\mathbf{o}}\bigg] \\
							  &= \sum_{a \in \calX_A} \delta^1_{a|\tilde{\mathbf{o}}} \sum_{\mathbf{s}, \mathbf{s'} \in \calX_S} p(\mathbf{s}'|\mathbf{s},a) p(\tilde{\mathbf{o}} | \mathbf{s}) p(\mathbf{s}) \bigg( r(\mathbf{s},a,\mathbf{s}') \\
							  &+ \bbE_{\boldsymbol{\delta}} \bigg[ \sum_{t=2}^T r(\mathbf{S}_t,A_t,\mathbf{S}_{t+1}) | \mathbf{o}_1 = \tilde{\mathbf{o}}, \mathbf{s}_2 = \mathbf{s}'\bigg] \bigg)\\
							  &= \sum_{a \in \calX_A} \delta^1_{a|\tilde{\mathbf{o}}} \alpha(a, \tilde{\mathbf{o}})
	\end{align*}
	where $$\alpha(a, \tilde{\mathbf{o}}) = \sum_{\mathbf{s}, \mathbf{s'} \in \calX_S} p(\mathbf{s}'|\mathbf{s},a) p(\tilde{\mathbf{o}} | \mathbf{s}) p(\mathbf{s}) \bigg( r(\mathbf{s},a,\mathbf{s}')+ \bbE_{\boldsymbol{\delta}} \bigg[ \sum_{t=2}^T r(\mathbf{S}_t,A_t,\mathbf{S}_{t+1}) | \mathbf{o}_1 = \tilde{\mathbf{o}}, \mathbf{s}_2 = \mathbf{s}'\bigg] \bigg)$$ Since all variables $(\mathbf{S}_t, \mathbf{O}_t, A_t)_{2 \leq t \leq T+1}$ are conditionally independent from $(S_1,O_1,A_1)$ given $S_2$ we have :
	$$\alpha(a, \tilde{\mathbf{o}}) = \sum_{\mathbf{s}, \mathbf{s'} \in \calX_S} p(\mathbf{s}'|\mathbf{s},a) p(\tilde{\mathbf{o}} | \mathbf{s}) p(\mathbf{s}) \bigg( r(\mathbf{s},a,\mathbf{s}')+ \bbE_{\boldsymbol{\delta}^{-1}} \bigg[ \sum_{t=2}^T r(\mathbf{S}_t,A_t,\mathbf{S}_{t+1}) | \mathbf{s}_2 = \mathbf{s}'\bigg] \bigg)$$
	where $\boldsymbol{\delta}^{-1} = \big( \delta^2, \ldots, \delta^T \big)$. Therefore, $\alpha$ does not depend on $\delta^1$. 
	Note that :
	\begin{align*}
		\tau_a^1 &= \sum_{s \in \calX_S^m,o \in \calX_O^m} \tau_{soa}^{1,m} \\
				 &= \sum_{s \in \calX_S^m} \tau_{s\tilde{o}^ma}^{1,m} \\
				 &= \delta^1_{a|\tilde{\mathbf{o}}} \sum_{s \in \calX_S^m, \overline{a} \in \calX_A} \tau_{s\tilde{o}^m\overline{a}}^{1,m} \\
				 &= \delta^1_{a|\tilde{\mathbf{o}}}
	\end{align*}
	where the last equality holds because the decision maker observes the initial observation. Therefore, we obtain :
	\begin{align*}
		\bbE_{\boldsymbol{\delta}} \bigg[ \sum_{t=1}^T r(s_t,a_t,s_{t+1}) | \mathbf{o}_1 = \tilde{\mathbf{o}}\bigg] &= \sum_{a \in \calX_A} \alpha(a,\tilde{\mathbf{o}}) \tau_a^1 \\
								&= \max_{a \in \calX_A} \alpha(a,\tilde{\mathbf{o}}) 
	\end{align*}
	where the last equality holds because the right-hand side is always greater than the left-hand side since $\sum_{a \in \calX_A} \tau_a^1 = 1$ and if $\tau_a^1 >0$ then $a \in \argmax_{a \in \calX_A} \alpha(a, \tilde{\mathbf{o}})$. Indeed, suppose that $\tau_a^1 > 0$ and $a \notin \argmax_{a \in \calX_A}  \alpha(a, \tilde{\mathbf{o}})$. Let $a^* \in \argmax_{a \in \calX_A}  \alpha(a, \tilde{\mathbf{o}})$, then the solution $\tilde{\tau}_a = 0$ and $\tilde{\tau}_{a^*} = \tau_{a^*} + \tau_a$. Therefore,
	$$\sum_{a \in \calX_A} \tau_a^1 \alpha(a,\tilde{\mathbf{o}}) < \sum_{a \in \calX_A} \tilde{\tau}_a^1 \alpha(a,\tilde{\mathbf{o}}) $$
	which contradicts the optimality assumption of $\tau$.
	We deduce that :
	\begin{align*}
	 	\bbE_{\boldsymbol{\delta}} \bigg[ \sum_{t=1}^T r(s_t,a_t,s_{t+1}) | \mathbf{o}_1 = \tilde{\mathbf{o}}\bigg] = \bbE_{\boldsymbol{\delta}^*} \bigg[ \sum_{t=1}^T r(s_t,a_t,s_{t+1}) | \mathbf{o}_1 = \tilde{\mathbf{o}}\bigg] 
	 \end{align*} 
	 because $\max_{a \in \calX_A} \tau_a^1 > 0$ since $\sum_{a \in \calX_A} \tau_a^1 = 1$.
\end{proof}

% section lp_cases (end)

% \inlAP{We also need to compare ouselves to \citet{khaled2013solving}}

% It seems that the state of the art is \citet{maua2012solving}

% \inlAP{Note that we focus on the directed graphical model version of \limid}